\documentclass[]{siamltex1213}
\usepackage{amssymb}
\usepackage{graphicx,latexsym,amsmath,color}

\renewcommand{\theequation}{\thesection.\arabic{equation}}
\numberwithin{equation}{section}
\renewcommand{\thetable}{\thesection.\arabic{table}}
\numberwithin{table}{section}
\renewcommand{\thefigure}{\thesection.\arabic{figure}}
\numberwithin{figure}{section}
\newcommand{\ith}{i\textsuperscript{th} }

\title{Tridiagonal Matrices and Boundary Conditions}

\author{J.J.P. Veerman \thanks{Maseeh Department of Mathematics \& Statistics, Portland State
University, Portland, OR 97201, USA}  
\and
David K. Hammond \thanks{Oregon Institute of Technology, Wilsonville, OR 97070}}

\begin{document}


\maketitle

\begin{abstract}
  We describe the spectra of certain tridiagonal matrices arising from
  differential equations commonly used for modeling flocking
  behavior. In particular we consider systems resulting from allowing
  an arbitrary boundary condition for the end of a one dimensional
  flock.  We apply our results to demonstrate how asymptotic stability
  for consensus and flocking systems depends on the imposed boundary
  condition.
\end{abstract}

\maketitle

\section{Introduction}
\label{chap:intro}

The $n+1$ by $n+1$ tridiagonal matrix
$$
A_{n+1}=\left(\begin{array}{cccccc}
b & 0 & 0 & & \\
a & 0 & c & 0 & \\
0 & a & 0 & c & \\
 & & \ddots & \ddots &  \ddots & \\
 & & & a & 0 & c \\
 & & & & a+e & d
\end{array}\right),
$$
is of special interest in many (high-dimensional) problems with \emph{local} interactions and
internal translation symmetry but with no clear preferred rule for the boundary condition.
We are interested in the spectrum and associated eigenvectors of this matrix.
In particular in Section \ref{chap:spectrum} we study how the spectrum
depends on choices for the boundary conditions implied by $d$ and $e$.

We will pay special attention to the following important subclass of these systems.

\begin{definition} \emph{If  $b=a+c$ and $c=e+d$, the matrix $A$ is called decentralized.}
\label{def:decentralized}
\end{definition}


One of the main applications of these matrices arises in the analysis
of first and second order systems of ordinary differential equations
in $\mathbb{R}^{n+1}$ such as

\begin{align}
\dot x = -L(x-h) &, \quad {\rm and } \label{eq:first-order}\\
 \ddot x=-\alpha L(x-h) -\beta L\dot x &.
\label{eq:second-order}
\end{align}

Here $L$ is the so-called directed graph Laplacian
(e.g. \cite{flocks2}), given by $L=D-A$, where $D$ is a diagonal
matrix with \ith entry given by the \ith row sum of $L$.
In the decentralized case $D=(a+c)I$, and $L$ is given simply by
\begin{equation}
L=bI-A=(a+c)I-A
\label{eq:Laplacian}
\end{equation}
In (\ref{eq:second-order}) $\alpha$ and $\beta$ are real numbers, and
$h$ is a constant vector with components $h_k$. Upon substitution of
$z\equiv x-h$, the fixed points of Equations \ref{eq:first-order} and
\ref{eq:second-order} are moved to the origin. 

It is easy to prove that the systems in Equations
\ref{eq:first-order} and
\ref{eq:second-order}
admit the solutions
\begin{equation}
x_k= x_0+h_k \quad {\rm and } \quad x_k= v_0t+x_0+h_k
\label{eq:coherent solutions}
\end{equation}
(for the first order system and the second order system, respectively) for arbitrary reals
$x_0$ and $v_0$ if and only if the system is decentralized.

The first order system given above is a simple model used to study
`consensus', while the second system models a simple instance of
`flocking' behavior. The latter is also used to study models for
automated traffic on a single lane road. The interpretation $d$ and
$e$ as specifying boundary conditions for these models can be
understood as follows. Following the transformation $z\equiv x-h$, we
may consider $z_k(t)$, for $0\leq k \leq n$, as the transformed
positions of the $n+1$ members of a ``flock''. Here $z_0(t)$ denotes
the transformed position of the \emph{leader}, the model
\ref{eq:first-order} then specifies that $\dot z_0 = 0$, $\dot z_k =
a(z_{k-1}-z_k) - c(z_k-z_{k+1})$ for $1\leq k < n$, and that $\dot z_n
= (a+e)z_{n-1} + dz_n$. The terms $a(z_{k-1} - z_k)$ and $-c
(z_k-z_{k+1})$ may be interpreted as control signals that are
proportional to the displacement of $z_k$ from its predecessor
$z_{k-1}$, and from its successor $z_{k+1}$. The equation giving $\dot
z_n$ is different as $z_n$ has no successor, selection of the boundary
condition consists of deciding what the behavior governing $z_n$
should be.  In the decentralized case, eliminating $d$ shows that
$\dot z_n = (a+e)(z_{n-1} - z_n)$, so that $\dot z_n$ is proportional
to the difference from its predecessor, and $e$ may be interpreted as
the additional amount of the proportionality constant due to boundary
effects. Interpretation of $d$ and $e$ for the second order system in
\ref{eq:second-order} is similar.
These problems are important examples of a more general class of
problems where oscillators are coupled according to some large
communication graph, and one wants to find out whether and how fast
the system synchronizes. The asymptotic stability of both systems is
discussed in Section \ref{chap:differential equations}.

One of the main motivations for this work came from earlier work
(\cite{Cantos2}) that led to the insight that in some important cases
changes of boundary conditions did not give rise to appreciable
changes in the dynamics of these systems (if the dimension was
sufficiently high).  This somewhat surprising discovery motivated the
current investigation into how eigenvalues change as a function of the
boundary condition. Indeed Corollaries \ref{cory:first-order} and
\ref{cory:second-order} corroborate that at least the asymptotic
stability of consensus systems and flock-formation systems is
unchanged for a large range of boundary conditions.

The method presented here relies on the observation that the
eigenvalue equation for $A$ can be rewritten as a two-dimensional
recursive system with appropriate boundary conditions. This procedure
was first worked out in \cite{FV}.  Here we give a considerably
refined version of that argument, that allows us to draw more general
conclusions. These conclusions are presented in Theorems
\ref{thm:-a.e.a}, \ref{thm:e>a}, and \ref{thm:e<-a}. The spectrum of
tridiagonal matrices has also been considered by Yueh \cite{Yueh} who
relies heavily on \cite{Elliott}.  In that work however the parameter
$e$ is zero, and the emphasis is on analyzing certain isolated cases,
while we attempt to give a comprehensive theory. The inclusion of the
parameter $e$ is necessary for our main application: decentralized
systems.  Related work has also been published by Willms
\cite{Willms2008}, who considered tridiagonal matrices where the
product of sub and super diagonal elements is constant, and Kouachi
\cite{Kouachi2006}, who considered a similar condition where the
product of sub and super diagonal elements alternates between two
values. Both of these conditions exclude the case when $e\neq 0$ in
the current work.


We assume $a$, $b$, $c$, $d$, $e$ to be real. The cases where $a=0$ or $c=0$ are very degenerate.
There are only 1 or 2 non-zero eigenvalues of $A$. We will not further discuss these cases.
That leaves $a>0$ and $c\neq 0$ as the general case to be studied. We will consider
$a>0$ and $c>0$ and will assume this unless otherwise mentioned.

In Section \ref{chap:prelim} we derive a polynomial whose roots will
eventually yield the eigenvalues. In the next Section we find the value of those roots.
Then in Section \ref{chap:spectrum} we use those results to characterize the spectrum
of $A$. In Section \ref{chap:decentralized} we apply this to the matrices associated with
decentralized systems. In Section \ref{chap:differential equations} we discuss the
consequences for the asymptotic stability of decentralized systems of ordinary differential
equations.

\section{Preliminary Calculations}
\label{chap:prelim}

We start by noting that $A_{n+1}$ is block lower triangular. One block has dimension 1 and eigenvalue
$b$. The other has dimension $n$. We begin by analyzing the eigenvector associated to this
eigenvalue.

\begin{proposition}
\emph{\bf i)}
If $b^2-4ac\neq 0$, then the eigenvector of $A$ associated to the eigenvalue
$b$ is $(v_0,\cdots v_n)$ where
\begin{equation}
v_k=x_+^k+c_-x_-^k
\end{equation}
and
\begin{equation}
x_{\pm}=\dfrac{1}{2c} \left( b \pm\sqrt{b^2-4ac}\right)
 \quad \textrm{and} \quad
 c_-=-\left(\dfrac{c}{a}\right)^n x_+^{2n-1}\,\dfrac{(a+e)+(d-b)x_+}{(c+e/a)x_++(d-b)}
\end{equation}
\emph{\bf ii)}
This associated eigenvector is the constant vector if and only $A$ is decentralized.\\
\emph{\bf iii)}
If $b^2-4ac=0$ the eigenvalues of $A$ can be found as limits of the eigenvalues
of case i.
\label{prop:eval=b}
\end{proposition}

\begin{proof}
We first prove i. The equation $Av=bv$ can be rewritten as
\begin{equation}
\forall \; j=1,\ldots, n-1\; : \;
\left(\begin{array}{c} v_j \cr v_{j+1}\end{array}\right)=C^j\left(\begin{array}{c} v_0 \cr
v_{1}\end{array}\right) \quad \mbox{and} \quad
(a+e)v_{n-1}+(d-b)v_n=0
\label{eq:recursion}
\end{equation}
Here the matrix $C$ is defined by
$$C= \left(\begin{array}{cc}
                  0 & 1 \cr
                  -\frac{a}{c} & \frac{b}{c}
\end{array}\right)\; ,
$$
The eigenvalues of $C$ are given by $x_\pm$ in the statement of the Proposition
and its associated eigenvectors are
$\left(\begin{array}{c} 1 \cr x_+\end{array}\right)$ and
$\left(\begin{array}{c} 1 \cr x_-\end{array}\right)$
\footnote{A simple calculation shows that $\left(\begin{smallmatrix} 0 \\ 1 \end{smallmatrix}\right)$ cannot be an eigenvector of $C$}.
Assuming that the eigenvalues $x_\pm$ are distinct, we can write $v_k=c_+x_+^k+c_-x_-^k$.
Of course $c_+$ can be chosen to be 1 without loss of generality.
Now $c_-$ is determined by substituting the expression for $v_k$ in the boundary condition
in Equation \ref{eq:recursion}.

We now prove ii. Substituting $v_j=1$ in Equation \ref{eq:recursion}, we immediately get
that $-a+b=c$ from the matrix equation and $a+e+d-b=0$, which imply the result.

Part iii follows from the fact that eigenvalues are continuous functions of the parameters.
\end{proof}

To facilitate calculations we set
$\tau^2\equiv a/c$, shave off the first row and column of $A_{n+1}$ and thus define the
reduced $n\times n$ matrix $Q_n$:
$$
Q_{n}=\left(\begin{array}{ccccc}
0 & a/\tau^2 & 0 & \\
a & 0 & a/\tau^2 & \\
  & \ddots & \ddots &  \ddots & \\
  & & a & 0 & a/\tau^2 \\
  & & & a+e & d
\end{array}\right),
$$

To find the spectrum of $Q$, we look for a number $r$ and a $n$-vector $v$ forming an eigenpair $(r,v)$ as follows
\begin{equation}
\begin{array}{cccc}
&\dfrac{a}{\tau^2}v_2 &=& rv_1\\
k\in \{2,\cdots, n-1\} & av_{k-1}+\dfrac{a}{\tau^2} v_{k+1} &= & rv_k\\
 & (a+e)v_{n-1}+dv_n &=& rv_n
\end{array}
\label{eq:recursion-general}
\end{equation}

We first collect a number of basic observations that allow us to deduce the main results
of this section.

\begin{lemma} Let $(r,v)$ an eigenpair for the matrix $Q$. Then
\begin{equation}
r=\sqrt{ac}(y+y^{-1}) \quad \mbox{and} \quad v_k=(\tau y)^k-\left(\dfrac{\tau}{y}\right)^k
\label{eq:eigenpair}
\end{equation}
where
\begin{equation}
\begin{array}{cc}
&a(y^{n+1}-y^{-(n+1)})-d\tau(y^n-y^{-n})-e(y^{n-1}-y^{-(n-1)})=0\\[0.2cm]
\mbox{or} & (ay^2-d\tau y-e)y^{n-1}+(ey^2+d\tau y-a)y^{-(n+1)}=0
\end{array}
\label{eq:polynomial}
\end{equation}
Furthermore, if we set $y=e^{i\phi}$ then
\begin{equation}
(e+a)\cos n\phi\,\sin \phi = (d\tau + (e-a)\cos \phi)\sin n\phi
\label{eq:cosine}
\end{equation}
If we assume that $\sin n\phi\neq 0$ then dividing by it gives
\begin{equation}
\cot n\phi \,\sin \phi = \dfrac{d\tau}{e+a}+\dfrac{e-a}{e+a}\cos \phi
\label{eq:cotangent}
\end{equation}
\label{lem:polynomial}
\end{lemma}

\begin{proof} We proceed as in Proposition \ref{prop:eval=b}.
The equation $Qv=rv$ can be rewritten as
\begin{equation}
\begin{array}{c}
\forall \; k\in\{1,\ldots, n\} \; ,
\left(\begin{array}{c} v_k \cr v_{k+1}\end{array}\right)=C^k\left(\begin{array}{c} v_0 \cr
v_{1}\end{array}\right) \\[0.3cm]
\mbox{and} \quad v_0=0 \quad \mbox{and} \quad \dfrac{a}{\tau^2}v_{n+1}-dv_n-ev_{n-1}=0
\end{array}
\label{eq:recursion2}
\end{equation}
Here the matrix $C$ is defined by
$$C= \left(\begin{array}{cc}
                  0 & 1 \cr
                  -\tau^2 & \frac{r\tau^2}{a}
\end{array}\right)\; ,
$$

The eigenvalues of $C$ are given by $x_\pm$ and their sum equals the trace of $C$, or:
\begin{equation}
\frac{r\tau^2}{a}=x_++x_-
\label{eq:trace}
\end{equation}
As before, if we assume that $x_+$ and $x_-$ are distinct, we have $v_k=x_+^k+c_-x_-^k$ but now $c_-$ is determined by the boundary condition $v_0=0$ in Equation \ref{eq:recursion}. This implies that $c_-=-1$. Now set
$x_\pm\equiv \tau y_\pm$. The product $x_-x_+$ equals the determinant of $C$, or $\tau^2$,
and therefore
\begin{equation}
y_-=y_+^{-1}
\label{eq:determinant}
\end{equation}
If we denote $y_+$ by $y$ we obtain Equation \ref{eq:eigenpair}.
The values of $y_\pm$ are now determined
by substituting $v_k$ into the last boundary condition of Equation \ref{eq:recursion2}.
We immediately get Equation \ref{eq:polynomial}.

Finally if we use
\begin{equation*}
2(y^{n+a}-y^{-(n+a)})=(y^{n}-y^{-n})(y^{a}+y^{-a})+
(y^{n}+y^{-n})(y^{a}-y^{-a})
\end{equation*}
and then substitute $y=e^{i\phi}$ in Equation \ref{eq:polynomial} we easily find
Equations \ref{eq:cosine} and \ref{eq:cotangent}.
\end{proof}

\begin{lemma} \emph{\bf i)}
If $1$ is a simple root of Equation \ref{eq:polynomial}, then $2\sqrt{ac}$ is not an eigenvalue. The same holds for $-1$ and $-2\sqrt{ac}$.\\
 \emph{\bf ii)} The set of eigenvalues is invariant under the transformations $\mathrm{inv}:
y\rightarrow y^{-1}$ and $\mathrm{conj}:y\rightarrow \bar{y}$ (the complex conjugate).
\label{lem:invariants}
\end{lemma}

\begin{proof} i: The procedure followed in the proof of Lemma \ref{lem:polynomial}
does not work if $x_+=x_-$. This happens if the discriminant of the characteristic polynomial
of the matrix $C$ is zero. This is only true if $r$, the eigenvalue of $Q$,
equals $\pm2\sqrt{ac}$, which is equivalent to $y=\pm1$.
On the other hand, if these roots are simple then
Equation \ref{eq:polynomial} has $2n$ other roots, which by Equation \ref{eq:determinant}
form $n$ pairs $y$ and $1/y$ (counting multiplicity). By Equation \ref{eq:eigenpair}
each pair yields an eigenvector.\\
ii: Both transformations leave Equation \ref{eq:polynomial} invariant.
\end{proof}

These two lemmas allow us in specific cases, namely when the polynomial in Equation
\ref{eq:polynomial} factors, to obtain a simple explicit representation of the eigenvalues.
Indeed in Yueh's paper \cite{Yueh} the emphasis is on these special cases.
We give a number of examples that are commonly used in the literature.
The remainder of the paper will then be devoted to obtain more general results.
We note that all three examples are special cases of Theorem \ref{thm:-a.e.a} part 2.

The first example is $c=a$ and $e=d=0$. The polynomial equation \ref{eq:polynomial} factors
as: $y^{2n+2}-1=0$. Thus after applying Lemmas \ref{lem:polynomial}
and \ref{lem:invariants} we see that the eigenvalues for $Q$ are given by
$2a\cos(\frac{\pi k}{n+1})$ for $k\in\{1,\cdots n\}$.
Our other two examples are of decentralized systems which are discussed in more detail
in the Section \ref{chap:decentralized}. The first of these is $c=a$ and $e=0$ and $d=a$.
The polynomial equation now reduces to: $(y-1)(y^{2n+1}+1)=0$. The roots
at $\pm1$ must again be ignored and the eigenvalues of $Q$ are
$2a\cos(\frac{\pi (2k-1)}{2n+1})$ for $k\in\{1,\cdots n\}$.
Finally we consider the case $c=a$ and $e=a$ and $d=0$. The polynomial equation becomes:
$(y^2-1)(y^{2n}+1)=0$. Eliminating $\pm1$ again, we get the eigenvalues
$2a\cos(\frac{\pi (2k-1)}{2n})$ for $k\in\{1,\cdots n\}$.

\section{The Roots of the Polynomial in Equation \ref{eq:polynomial}}
\label{chap:polynomial}

\begin{proposition} If $a+e=0$ then
\begin{equation*}
y_\ell = e^{\frac{\pi i \ell}{n}}\;\mbox{for}\;\ell\in\{1,\cdots n-1\}\; ,\quad
y_n=\frac{1}{2a} \left(d\tau\pm \sqrt{d^2\tau^2-4a^2}\right)
\end{equation*}
Furthermore
\begin{equation*}
\begin{array}{cccc}
1.&2a < d\tau  &\mathrm{then}& y_n>1\\
2.&-2a\leq d\tau \leq 2a &\mathrm{then}& |y_n|=1 \\
3.&d\tau <-2a &\mathrm{then}& y_n<-1
\end{array}
\end{equation*}
\label{prop:a+e=0}
\end{proposition}

\begin{proof} Equation \ref{eq:polynomial} factors to become:
\begin{equation*}
(y^2-\frac{d\tau}{a}y+1)(y^{n-1}-y^{-(n+1)})=0
\end{equation*}
The roots at $\pm1$ can be discarded because of Lemma \ref{lem:invariants}The remaining roots are as stated in the Proposition.
\end{proof}

\begin{definition}
\emph{The symbol $\phi_\ell$ means a solution of Equation
\ref{eq:cotangent} in the interval $\left(\frac{(\ell-1)\pi}{n}, \frac{\ell\pi}{n}\right)$.
The notation $y_\ell\approx L$ is reserved for:
there is $\kappa>1$ such that $y_\ell-L=O(\kappa^{-n})$ as $n$ tends to $\infty$.
We will furthermore denote the roots of $ay^2-d\tau y-e$ as follows:
\begin{equation*}
y_\pm \equiv \frac{1}{2a}\left(d\tau\pm\sqrt{d^2\tau^2+4ae}\right)
=\frac{1}{2\sqrt{ac}}\left(d\pm\sqrt{d^2+4ce}\right)
\end{equation*}}
(We will choose the branch-cut for the root as the positive imaginary axis. So $\sqrt{x}$
will always have a non-negative real part.)
\label{def:phi_ell-etc}
\end{definition}

\begin{proposition} Let $-a < e \leq a$ and $a$, $c$, $d$, $e$ fixed.
Then (not counting simple roots at $\pm 1$) for $n$ large enough we have the following
solutions of Equation \ref{eq:polynomial}:
\begin{equation*}
\begin{array}{cccl}
1.& a-e < d\tau & : & y_\ell=e^{i\phi_\ell}\; \mathrm{ for } \; \ell\in\{2,\cdots n\}
\;\mathrm{ and }\;y_1\approx y_+ > 1 \\[.15cm]
2.&-(a-e)\leq d\tau \leq a-e & : & y_\ell=e^{i\phi_\ell}\; \mathrm{ for }
\; \ell\in\{1,\cdots n\} \\[.15cm]
3. & d\tau <-(a-e)  & : & y_\ell=e^{i\phi_\ell}\; \mathrm{ for } \; \ell\in\{1,\cdots n-1\}
\;\mathrm{ and }\; y_n\approx y_- < -1
\end{array}
\end{equation*}
\label{prop:-a<e<a}
\end{proposition}

\begin{proof} In this case $\sin n\phi=0$ and $\sin \phi\neq 0$
in Equation \ref{eq:cosine} does not give any solutions.
We are thus allowed to divide by $\sin n\phi=0$ to obtain Equation \ref{eq:cotangent}.
The right hand of that equation consists of $n$ smooth decreasing branches (see Appendix 1)
on
$$
\bigcup_{\ell=1}^n\,I_\ell \equiv\left[ 0,\frac \pi n \right)\cup\Big(\frac \pi n, \frac{2\pi}{n}\Big)\cdots
\cup \Big( \frac{(n-2)\pi}{n}, \frac{(n-1)\pi}{n} \Big) \cup \Big(\frac{(n-1)\pi}{n}, \pi\Big]
$$
whose ranges are $(-\infty,\frac 1n]$ on $I_1$, $[-\frac 1n,\infty)$ on $I_n$, and $(-\infty,\infty)$ in all other cases.
The left hand is non-decreasing on $[0,\pi]$. Thus every interval $I_\ell$ has a root, except
possibly the first and the last.

In each of the three cases we first solve Equation \ref{eq:cotangent} (see Figure
\ref{fig:a.between.e}). Then we find any remaining roots by other means.
If $d\tau >a-e$ (ie in case 1 of the Proposition) then the right hand at $\phi=0$
of that equation is positive. This means that for $n$ large enough there is no root in the
interval $I_1$ of Equation \ref{eq:cotangent}. All other intervals
$I_\ell$ (ie: $\ell\in \{2,\cdots n\}$ contain a root $y_\ell=e^{i\phi_\ell}$.
If $ay^2-d\tau y-e$ has a root has a root $y_+$ with absolute value greater than 1, then by
Lemma \ref{lem:App2} part ii we know that for large $n$ Equation \ref{eq:polynomial}
has a root exponentially close (as $n$ tends to infinity) to $y_+$ (of Definition
\ref{def:phi_ell-etc}). By substituting $d\tau>a-e$ in $y_+$ one sees that the root is real
and greater than 1. (Observe that by Lemma \ref{lem:invariants} this also yields a root
exponentially close to the reciprocal of $y_+$ and which is not equal to $y_-$.)
We thus obtain $n$ roots $y_\ell$ and $y_+$. The other $n$ roots are given by their
reciprocals. These combine (by Equation \ref{eq:eigenpair}) to give $n$ eigenvalues.
The (simple) roots at $\pm 1$ are discarded by Lemma \ref{lem:invariants}.

In case 2 we have that $d\tau \leq a-e$ and so $I_1$ contains a root
of Equation \ref{eq:cotangent}, and we have that $d\tau\geq -(a-e)$ and so the
interval $I_n$ also contains a root. Thus all non-trivial roots have the form
$e^{i\phi_\ell}$.

In case 3 we have roots in all $I_\ell$ except in $I_n$. An almost identical argument to
the one above now shows that there is a root exponentially close to $y_-$
which is real and smaller than -1.
\end{proof}

\begin{figure}[pbth]
\center
\includegraphics[width=5.0in]{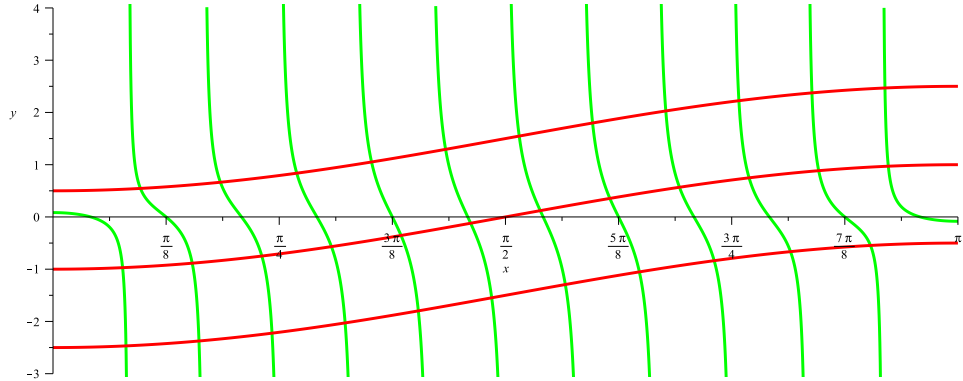}
\caption{ \emph{The three cases of Proposition \ref{prop:-a<e<a}: The fast varying (green)
plot is the right hand of Equation \ref{eq:cotangent}, and the slow varying (red)
curves represent its right hand side in case 1, 2 , and 3 (from top to bottom).}}
\label{fig:a.between.e}
\end{figure}

\begin{proposition} Let $e> a$ and $a$, $c$, $d$, $e$ fixed. Then (not counting simple roots
at $\pm 1$) for $n$ large we have the following solutions of Equation \ref{eq:polynomial}:
\begin{equation*}
\begin{array}{cccl}
1.&  -(a-e)\leq d \tau & : & y_\ell=e^{i\phi_\ell}\; \mathrm{ for } \; \ell\in\{2,\cdots n\}
\;\mathrm{ and }\; y_1\approx y_+ \geq \frac ea\\[.15cm]
2.&(a-e)< d\tau < -(a-e) & : & y_\ell=e^{i\phi_\ell}\; \mathrm{ for } \; \ell\in\{2,\cdots n-1\}
 \;\\[.1cm]
 &&&   \mathrm{ and }\;y_1\approx y_+\in[1,\frac ea)  \;, \; y_n \approx y_-\in (-\frac ea, -1]\\[.15cm]
3. & d\tau \leq(a-e)  & : & y_\ell=e^{i\phi_\ell}\; \mathrm{ for } \; \ell\in\{1,\cdots n-1\} \;\mathrm{ and }\; y_n\approx y_- \leq\frac{-e}{a}
\end{array}
\end{equation*}
\label{prop:e>a}
\end{proposition}

\begin{proof} The proof is similar to that of the previous Proposition. There
are again three cases (see Figure \ref{fig:e.greater.a}). Note that
$0\leq\frac{e-a}{e+a}\leq 1$, so that the slopes of both right hand and left hand
of Equation \ref{eq:cotangent} are negative. Corollary \ref{cor:App1} insures that
on each branch $I_\ell\in [0,\pi]$
\begin{equation*}
\cot(n\phi)\sin(\phi)-\frac{e-a}{e+a}\cos(\phi) \mathrm{\;is\;decreasing\;and\;}
\dfrac{\partial}{\partial \phi}
\left[ \cot(n\phi)\sin(\phi)-\frac{e-a}{e+a}\cos(\phi)\right]\leq 0
\end{equation*}
So it can only have one zero in each branch.

With this proviso, cases 1 and 3 can be resolved as in Proposition \ref{prop:-a<e<a}.
For example, if the left hand of Equation \ref{eq:cotangent} at $\phi =\pi$ is greater than
or equal to zero then there for all $n$ is a root in the interval $I_n$ but
(for $n$ large enough) none in $I_1$. This happens if $-(a-e)\leq d \tau$, which gives case 1.
Case 3 is similar.
It remains to analyze case 2. Denote $f(y)\equiv ay^2-d\tau y-e$. Since
$f(-1)$, $f(0)$, and $f(1)$ are negative $f$ must have one root greater than 1,
and one less than -1. By Lemma \ref{lem:App2} each of these roots is approximated
exponentially well by a root of Equation \ref{eq:polynomial}. Since furthermore
$f(-\frac ea)$ and $f(\frac ea)$ are greater than 0, the loci of these roots are
(for large enough $n$), one in $(1,\frac ea)$ and one in $(-\frac ea, -1)$.
\end{proof}

\begin{figure}[pbth]
\center
\includegraphics[width=5.0in]{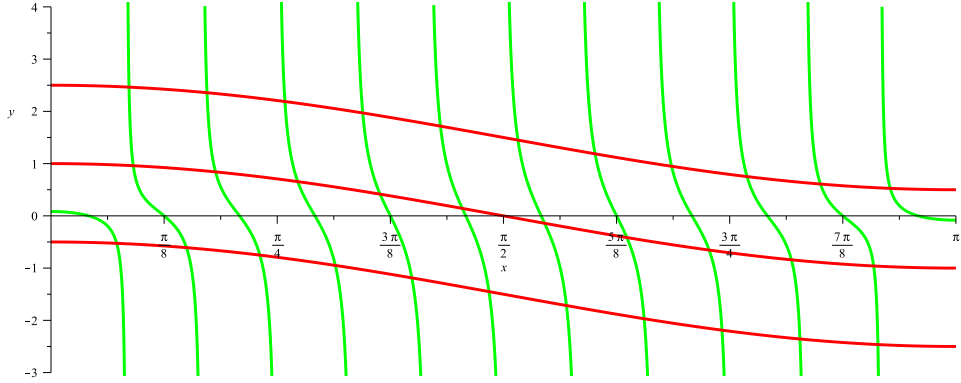}
\caption{ \emph{The three cases of Proposition \ref{prop:e>a} (1, 2, 3 from top to bottom).
See caption of Figure \ref{fig:a.between.e}.}}
\label{fig:e.greater.a}
\end{figure}

\begin{proposition} Let $e<-a$ and $a$, $c$, $d$, $e$ fixed. Then (not counting simple roots
at $\pm 1$) for $n$ large we have the following solutions of Equation \ref{eq:polynomial}:
\begin{equation*}
\begin{array}{cccl}
1.& d\tau \leq -(a-e) &\mathrm{ then }& y_\ell=e^{i\phi_\ell}\; \mathrm{ for } \; \ell\in\{2,\cdots n\} \;\mathrm{ and }\;
y_1\approx y_-\leq \frac ea\\[.15cm]
2.&-(a-e)< d\tau < a-e &\mathrm{then}&
y_\ell=e^{i\phi_\ell}\; \mathrm{ for } \; \ell\in\{2,\cdots n-1\}
\;\mathrm{ and }\; |y_1|,|y_n|>1\\[.15cm]
3. &  (a-e) \leq d\tau  &\mathrm{then}& y_\ell=e^{i\phi_\ell}\; \mathrm{ for } \; \ell\in\{1,\cdots n-1\} \;\mathrm{ and }\;
y_n\approx y_+ \geq -\frac ea
\end{array}
\end{equation*}
Case 2 can be further subdivided as follows:
\begin{equation*}
\begin{array}{cccl}
2a. & -(a-e) < d\tau\leq -2\sqrt{|ae|} & : & y_1\approx \tilde y_1, y_n\approx \tilde y_n,
\;\tilde y_1, \tilde y_n \in \left[\frac ea,-1\right]\\
2b. & -2\sqrt{|ae|} < d\tau < 2\sqrt{|ae|} & : & y_1, y_n
\mathrm{\;not \;real, \;} |y_1|\approx -\frac ea\approx |y_n|\\
2c. & 2\sqrt{|ae|} \leq d\tau < (a-e) & : & y_1\approx \tilde y_1, y_n\approx \tilde y_n,
\;\tilde y_1, \tilde y_n \in \left[1,-\frac ea\right]
\end{array}
\end{equation*}
\label{prop:e<-a}
\end{proposition}

\begin{proof} Dividing by $e+a$ (which is negative) we see that the three cases in the statement of the Proposition are equivalent to:
\begin{equation*}
\begin{array}{cl}
1. & \dfrac{e-a}{e+a}\leq \dfrac{d\tau}{e+a}\\[0.25cm]
2. & -\dfrac{e-a}{e+a}<\dfrac{d\tau}{e+a}<\dfrac{e-a}{e+a}\\[0.25cm]
3. & \dfrac{d\tau}{e+a}\leq \dfrac{e-a}{e+a}\\
\end{array}
\end{equation*}
As in the proof of Proposition \ref{prop:e>a}, these three cases correspond to Equation
\ref{eq:cotangent} not having a solution in $I_1$ (case 1), not having having solutions in
both $I_1$ and $I_n$ (case 2), and not having a solution in $I_n$ (in case 3). Thus in
case 1 we need to determine a solution $y_1$ of Equation \ref{eq:polynomial} that is
not on the unit circle. We need a similar solution $y_n$
in case 3. In case 2, we need to determine two extra solutions $y_1$, $y_n$.

Setting $f(y)\equiv ay^2-d\tau y - e$ it is straightforward to verify the following table
of values of $f$ for the three cases mentioned above.
\begin{equation*}
\begin{array}{cc|ccccccc}
 & & y=\frac{e}{a} & & y=-1 & y=0 & y=1 & & y=\frac{-e}{a}\\
\hline \hline
1. & d\tau\leq -(a-e)  & \leq 0 & & \leq 0 & >0 & >0 & & >0 \\
\hline
2a. & -(a-e)<d\tau < -2\sqrt{|ae|} & >0 & f\left(\frac{d\tau}{2a}\right)\leq 0 & >0 & >0 & >0 & & >0\\
\hline
2b. & -2\sqrt{|ae|}<d\tau\leq 2\sqrt{|ae|} & >0 & & >0 & >0 & >0 & & >0\\
\hline
2c. & 2\sqrt{|ae|}\leq d\tau<(a-e) & >0 & & >0 & >0 & >0 & f\left(\frac{d\tau}{2a}\right)\leq 0 & >0\\
\hline
 3. & d\tau \geq (a-e) & >0 & & >0 & >0 & \leq 0 & & \leq 0 \\
\hline
\end{array}
\end{equation*}
From this table it is clear that in case 1 $f$ has one real root less than or equal to
$\frac ea$. Similar in case 3 where there is root greater than or equal to $\frac{-e}{a}$.
In cases 2a and 2c there are two real valued solutions with absolute
value greater than 1. Finally in case 2b the roots of $f$ are complex conjugates with product
$-e/a$. By hypothesis this is greater than 1. So also here $f$ has two roots with absolute
value greater than 1.

By Lemma \ref{lem:App2} each of the larger than unity roots of $f$ is approximated
exponentially
well (in $n$) by a root of Equation \ref{eq:polynomial}. This gives the final result.
\end{proof}

It perhaps worth pointing out that in the last case it is now \emph{not true} that
Equation \ref{eq:cotangent} has at most one solution in
each interval $I_\ell$. See for instance Figure \ref{fig:e.less.-a} where one can see
3 solutions in $I_2$. However, for large $n$, as the above argument shows, it \emph{is true}
that these solutions are unique. A more direct proof of this fact appears complicated.
\begin{figure}[pbth]
\center
\includegraphics[width=5.0in]{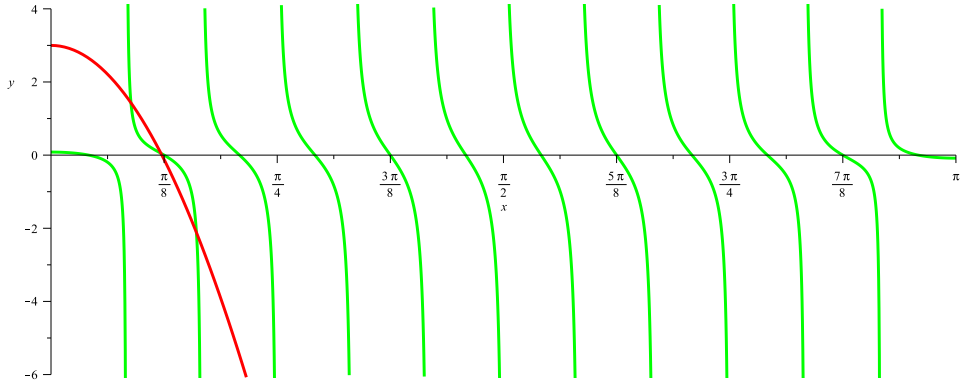}
\caption{ \emph{Detail of case 2b of Proposition \ref{prop:e<-a}: $a=1$, $e=-1.05$, and
$d\tau=-1.9$. Here $n=12$. The multiple roots in the second branch disappear for large $n$
(while holding other parameters fixed).}}
\label{fig:e.less.-a}
\end{figure}

\section{The Spectra}
\label{chap:spectrum}

In this section we apply Equation \ref{eq:eigenpair} of Lemma \ref{lem:polynomial} to
the propositions of the previous section
to obtain the spectrum of the $n+1$ by $n+1$ matrix $A$ of Section \ref{chap:intro}. This
gives us our main results. About the associated eigenvectors we remark here that those
can be obtained using the same lemma. Note that wherever there is a double root in the
polynomial equation equation \ref{eq:polynomial} we obtain only one eigenvector.
A generalized eigenvector (associated with a Jordan normal block of dimension 2 or higher)
can be derived (see Yueh \cite{Yueh} for some examples). Since we are mainly interested in
the spectrum we will not pursue this here.

\begin{definition} \emph{In this section we will denote, for $e\neq 0$,
\begin{equation} \label{eq:rpm_def}
r_\pm \equiv \dfrac{1}{2}\left[\left(1-\dfrac ae\right)d \pm \left(1+\dfrac ae\right)\sqrt{d^2+4ce} \right].
\end{equation}
When $e=0$, by taking limits as $e\to 0$ we define\footnote{Note that
  $r_{+}$ (respectively $r_{-}$) is not defined if $e=0$ and $d<0$
  (respectively $d>0$); careful examination of the cases corresponding
  to these parameter values shows that the undefined symbols are not
  used.}
\begin{align*}
r_{-} &= d+\frac{ac}{d} \mbox {\quad if $d<0$ } \\
r_{+} &= d+\frac{ac}{d} \mbox {\quad if $d>0$} .
\end{align*}
In this section the symbol $\psi_\ell$ means a solution of Equation
\ref{eq:cotangent} in the interval $\left[\frac{(\ell-1)\pi}{n}, \frac{\ell\pi}{n}\right)$
(cf Definition \ref{def:phi_ell-etc}).}
\label{def:psi_ell}
\end{definition}

\begin{theorem} Let $-a \leq e \leq a$ and $a$, $c$, $d$, $e$ fixed. Then for $n$ large enough the $n+1$ eigenvalues
$\{r_i\}_{i=0}^n$ of the matrix $A$ are the following.
First, $r_0=b$. The other $n$ eigenvalues are:
\begin{equation*}
\begin{array}{cccl}
1.& (a-e)\sqrt{\frac{c}{a}} < d &:& r_\ell=2\sqrt{ac}\cos \psi_\ell\; ,\; \ell\in\{2,\cdots n\} ,\; r_1\approx r_+>2\sqrt{ac}\\[.15cm]
2.&-(a-e)\sqrt{\frac{c}{a}}\leq d \leq (a-e)\sqrt{\frac{c}{a}} &:& r_\ell=2\sqrt{ac}\cos \psi_\ell\; , \; \ell\in\{1,\cdots n\} \\[.15cm]
3. & d <-(a-e)\sqrt{\frac{c}{a}}  &:& r_\ell=2\sqrt{ac}\cos \psi_\ell\; ,\; \ell\in\{1,\cdots n-1\}
,\; r_n\approx r_-<-2\sqrt{ac}
\end{array}
\end{equation*}
When $a=-e$ then $\psi_\ell=(\ell-1)\pi/n$, otherwise $\psi_\ell \in ((\ell-1)\frac{\pi}{n}, \ell\frac{\pi}{n})$ (except possibly for $\ell=1$ and $\ell=n$).
\label{thm:-a.e.a}
\end{theorem}

\begin{proof} This follows from applying Proposition \ref{lem:polynomial} to
Propositions \ref{prop:a+e=0} and \ref{prop:-a<e<a}. We have substituted $a/c$ for $\tau^2$.
$r_+$ and $r_-$ are obtained by simplifying the expressions for $\sqrt{ac}(y_++y_+^{-1})$ and
$\sqrt{ac}(y_-+y_-^{-1})$, respectively.
\end{proof}

The next two results follow in the same manner from Propositions \ref{prop:e>a} and \ref{prop:e<-a}.
We omit the proofs since they are easy.

\begin{theorem} Let $e>a$ and $a$, $c$, $d$, $e$ fixed. Then for $n$ large enough the $n+1$ eigenvalues $\{r_i\}_{i=0}^n$ of the matrix $A$ are the following. First, $r_0=b$. The other $n$ eigenvalues are:
\begin{equation*}
\begin{array}{cccl}
1.& -(a-e)\sqrt{\frac{c}{a}} \leq d &:& r_\ell=2\sqrt{ac}\cos \psi_\ell\; ,\; \ell\in\{2,\cdots n\}
,\; r_1\approx r_+\\[.15cm]
2.&(a-e)\sqrt{\frac{c}{a}}<d < -(a-e)\sqrt{\frac{c}{a}} &:& r_\ell=2\sqrt{ac}\cos \psi_\ell\; , \; \ell\in\{2,\cdots n-1\},\; r_1\approx r_+,  r_n\approx r_-\\[.15cm]
3. & d \leq (a-e)\sqrt{\frac{c}{a}}  &:& r_\ell=2\sqrt{ac}\cos \psi_\ell\; ,\; \ell\in\{1,\cdots n-1\}
,\; r_n\approx r_-
\end{array}
\end{equation*}
Furthermore we also have that in these cases:
\begin{equation*}
\begin{array}{cl}
1.& r_+ \geq \sqrt{ac}\left(\dfrac ea + \dfrac ae\right)\\[.15cm]
2.&r_- \in \left(-\sqrt{ac}\left(\dfrac ea + \dfrac ae\right), -2\sqrt{ac}\right]\;  \mathrm{ and } \;
r_+\in \left[2\sqrt{ac}, \sqrt{ac}\left(\dfrac ea + \dfrac ae\right)\right)\\[.15cm]
3. & r_- \leq -\sqrt{ac}\left(\dfrac ea + \dfrac ae\right)
\end{array}
\end{equation*}
\label{thm:e>a}
\end{theorem}

\begin{theorem} Let $e<-a$ and $a$, $c$, $d$, $e$ fixed. Then for $n$ large enough the $n+1$ eigenvalues $\{r_i\}_{i=0}^n$ of the matrix $A$ are the following. First, $r_0=b$. The other $n$ eigenvalues are:
\begin{equation*}
\begin{array}{cccl}
1.& d \leq -(a-e)\sqrt{\frac{c}{a}} &:& r_\ell=2\sqrt{ac}\cos \psi_\ell\; ,\; \ell\in\{2,\cdots n\}
,\; r_1\approx r_-\\[.15cm]
2.&-(a-e)\sqrt{\frac{c}{a}} < d < (a-e)\sqrt{\frac{c}{a}} &:& r_\ell=2\sqrt{ac}\cos \psi_\ell\; , \; \ell\in\{2,\cdots n-1\},\; r_1\approx r_+,  r_n\approx r_-\\[.15cm]
3. & (a-e)\sqrt{\frac{c}{a}} \leq d &:& r_\ell=2\sqrt{ac}\cos \psi_\ell\; ,\; \ell\in\{1,\cdots n-1\}
,\; r_n\approx r_+
\end{array}
\end{equation*}
Furthermore we also have that in these cases:
\begin{equation*}
\begin{array}{cccl}
1.  & d \leq -(a-e)\sqrt{\frac{c}{a}} &:& r_-\leq \sqrt{ac}\left(\dfrac ea + \dfrac ae\right)\\[.15cm]
2a. & -(a-e)\sqrt{\frac{c}{a}} < d\leq -2\sqrt{|ce|} &:& r_-,r_+\in \left(\sqrt{ac}\left(\dfrac ea + \dfrac ae\right),-2\sqrt{ac} \right]\; \\[.15cm]
2b. & -2\sqrt{|ce|} < d < 2\sqrt{|ce|} &:&  r_-,r_+ \;\mathrm{not \; real}\\[.15cm]
2c. & 2\sqrt{|ce|}\leq d < (a-e)\sqrt{\frac{c}{a}} & : &
r_-,r_+ \in \left[2\sqrt{ac}, -\sqrt{ac}\left(\dfrac ea + \dfrac ae\right)\right)\\[.15cm]
3.  & (a-e)\sqrt{\frac{c}{a}}\leq d &:&  r_+ \geq -\sqrt{ac}\left(\dfrac ea + \dfrac ae\right)
\end{array}
\end{equation*}
\label{thm:e<-a}
\end{theorem}

As an illustration of these ideas we plot the solutions of Equation \ref{eq:polynomial} and
the spectrum of $A$ in the case where $a=c=1$, $e=-9/4$. We take $d\in\{2.95,3.05,3.3\}$ so that
(since $\tau=1$) we are in case 2b, 2c, and 3 respectively of Proposition \ref{prop:e<-a}
and Theorem \ref{thm:e<-a}. The results can be found in Figure \ref{fig:illustration}.
These results are numerical: the first three figures were obtained with the MAPLE ``fsolve" routine,
the last three were obtained from the former by applying Equation \ref{eq:eigenpair} to the roots to get the eigenvalues.
We took $n=100$. The eigenvalue $b$ of the matrix $A$ is not displayed.
\begin{figure}[pbth]
\center
\includegraphics[width=2in]{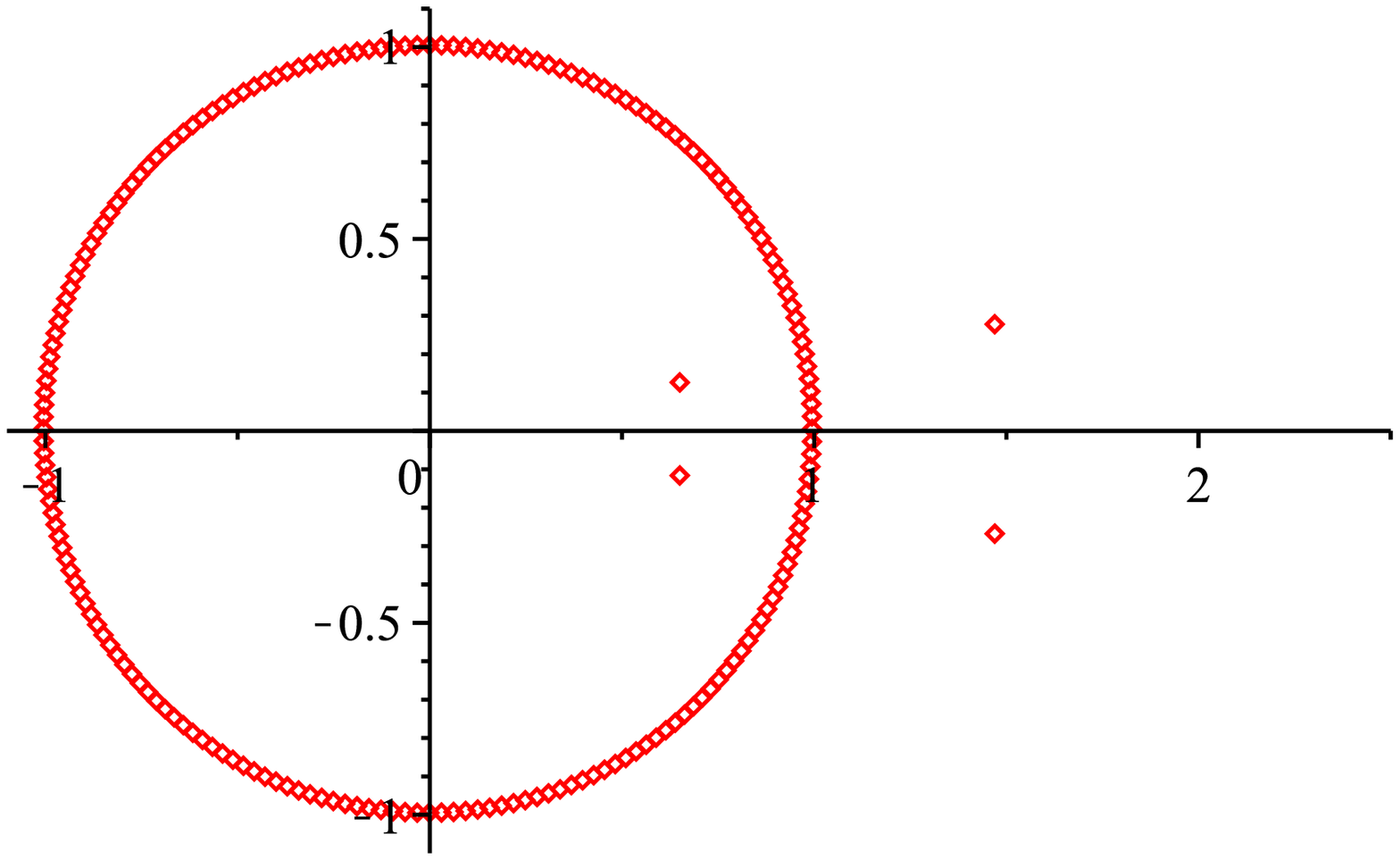}
\includegraphics[width=2in]{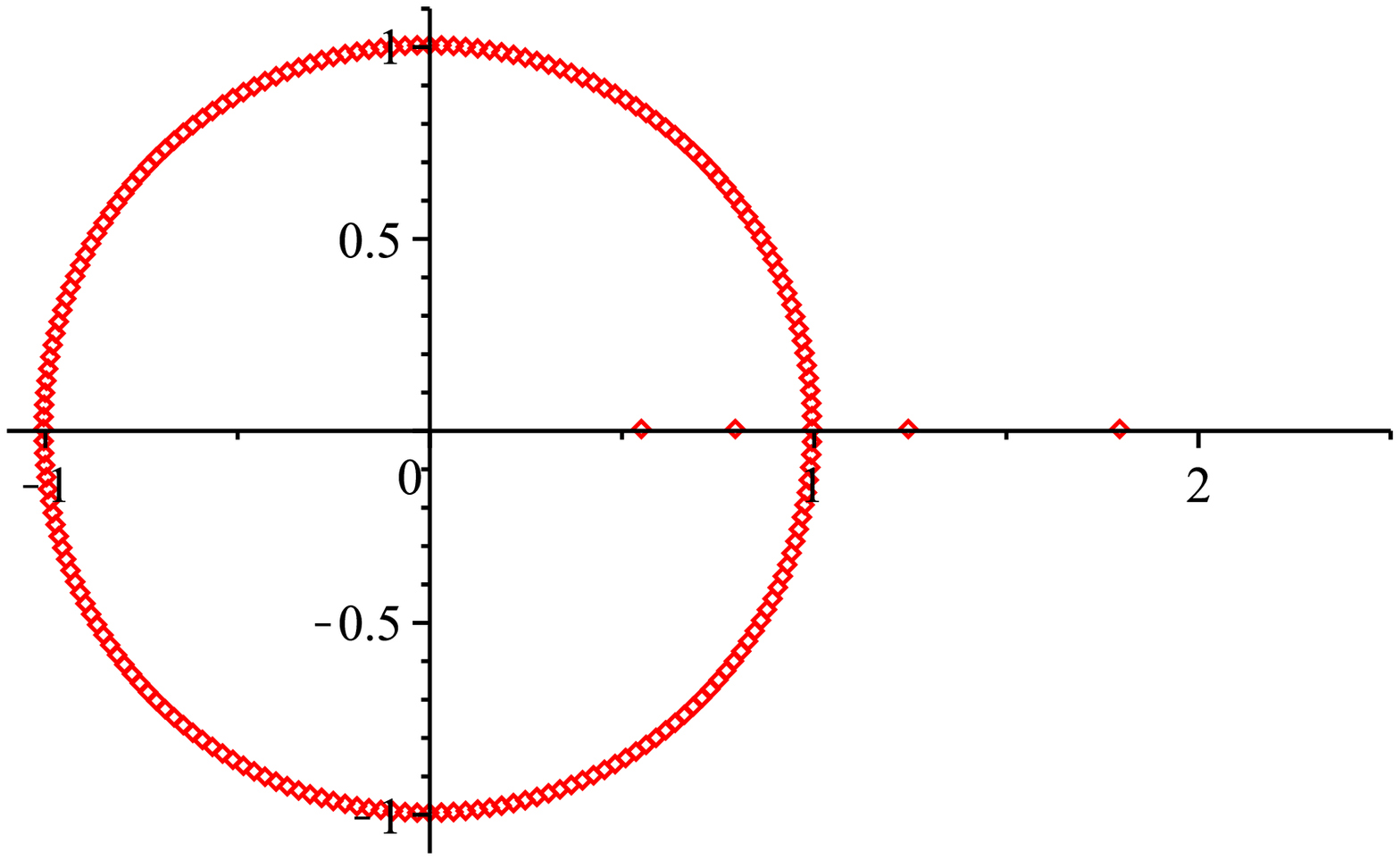}
\includegraphics[width=2in]{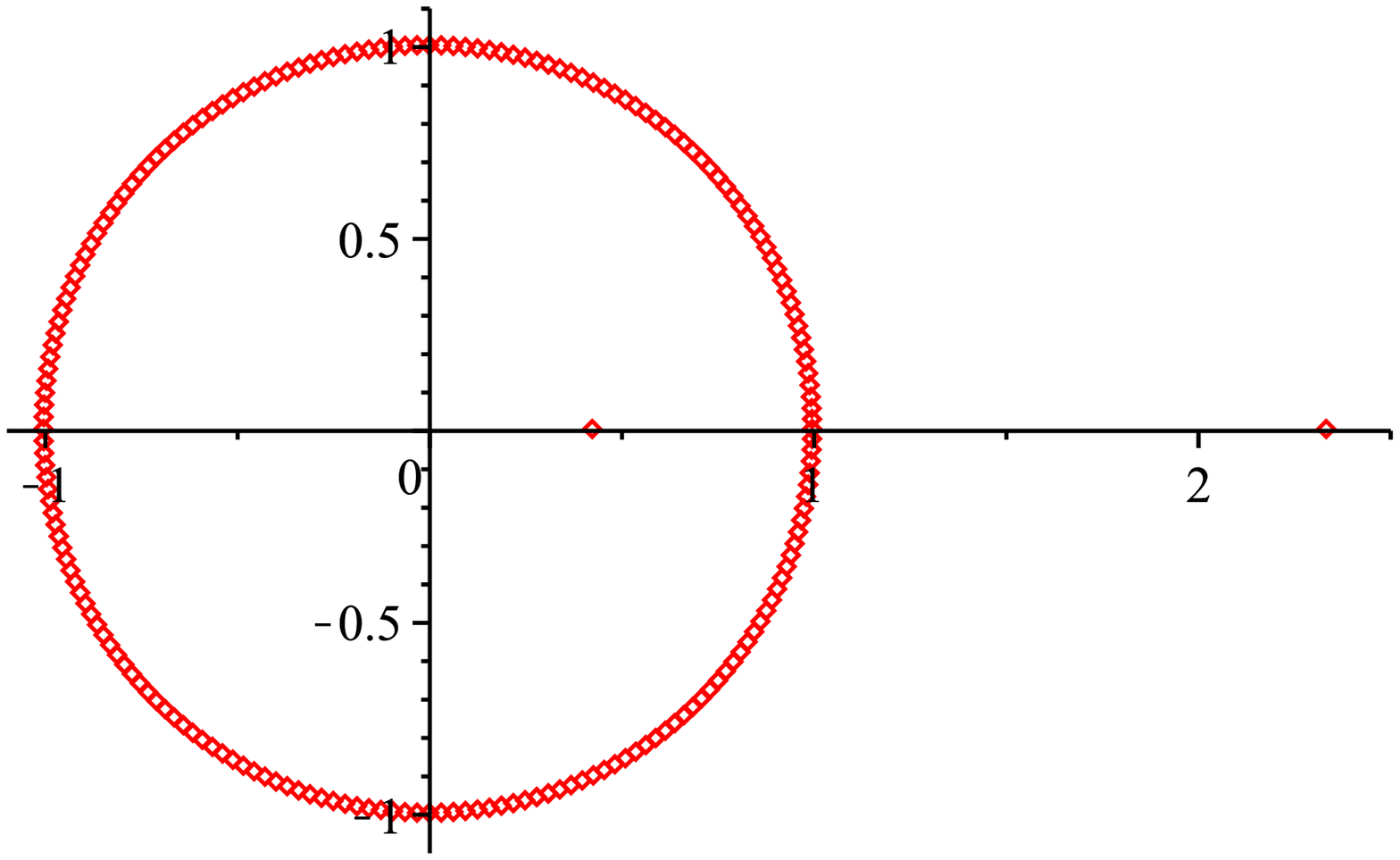}

\includegraphics[width=2in]{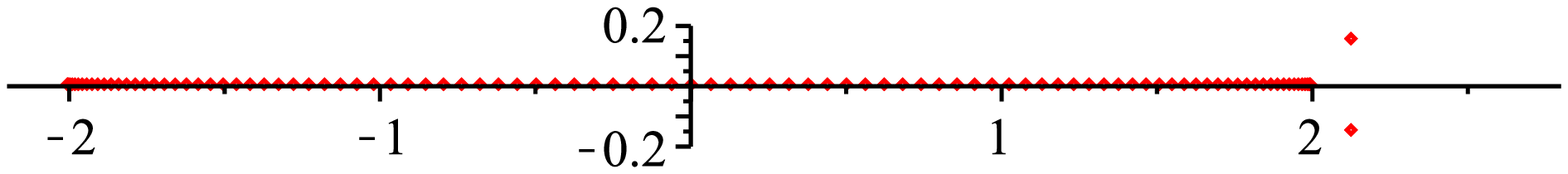}
\includegraphics[width=2in]{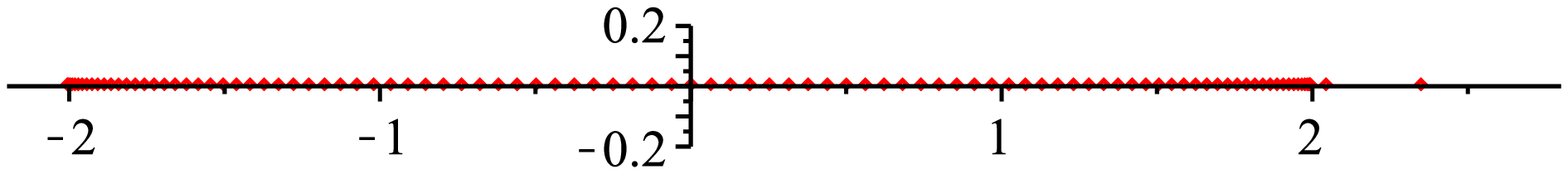}
\includegraphics[width=2in]{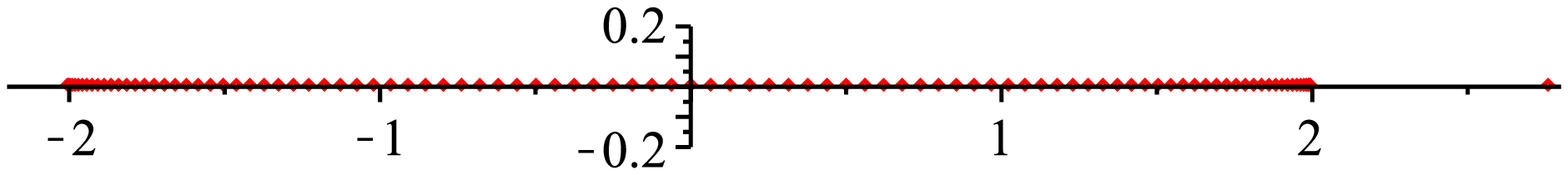}
\caption{ \emph{Numerical illustration of the roots of the polynomial
    (\ref{eq:polynomial}) (top) and the eigenvalues (bottom), for
    three different parameter values (given in text)}
\label{fig:illustration}
}
%
\end{figure}

\section{The Decentralized Case}
\label{chap:decentralized}

We look at the decentralized case defined in Definition \ref{def:decentralized}.

\begin{lemma} In the decentralized case the eigenvalues $r_\pm$ of Definition
\ref{def:psi_ell} become
\begin{align*}
r_+ = a+c \quad \mathrm{and} \quad r_-=-\left(\dfrac{ac}{e}+e\right)
& \text{, \quad if $c+e>0$ }\\
r_+ =  -\left(\dfrac{ac}{e}+e\right) \quad \mathrm{and} \quad r_-=a+c
& \text{, \quad if $c+e<0$ }
\end{align*}
\label{lem:special-evals}
\begin{proof}
Substituting $d=c-e$ into (\ref{eq:rpm_def}) gives $r_\pm =
\frac{1}{2}\left[(1-\frac{a}{c})(c-e) \pm
  (1+\frac{a}{e})\sqrt{(c+e)^2}\right]$, using $\sqrt{(c+e)^2}=|c+e|$ and
    simplifying gives the result.
\end{proof}
\end{lemma}

\begin{theorem} Let $a$, $c$, $d$, $e$ fixed so that $A$ is decentralized.
Then for $n$ large enough the $n+1$ eigenvalues
$\{r_i\}_{i=0}^n$ of the matrix $A$ are the following. First, $r_0=a+c$. $n-2$ eigenvalues are $r_\ell=2\sqrt{ac}\cos \psi_\ell$ for $\ell\in\{2,\cdots n-1\}$ where the $\psi_\ell$ are
solutions of Equation \ref{eq:cotangent}. The remaining two eigenvalues
($\ell=1$ and $\ell=n$) either also satisfy that formula or else are exponentially close
(in $n$) to the ones given in the table below. In the table we list the domain of $e$
left of the colon, and the appropriate special eigenvalues (if any) right of the colon.
{\small
\begin{equation*}
\setlength\arraycolsep{3pt}
\begin{array}{|c||c|c|c|}
\hline
---& 0<c<a & c=a & a<c \\
\hline\hline
e<-a & \begin{array}{cc} (-\infty,-a) :& -\left(\frac{ac}{e}+e\right) \end{array}
& \begin{array}{cc} (-\infty,-a) :& -a\left(\frac{a}{e}+\frac ea\right) \end{array}
& \begin{array}{cc} \left(-\infty, -\sqrt{ac}\right) :&  \left\{ a+c, -\left(\frac{ac}{e}+e\right)\right\}\\
\left[-\sqrt{ac},-a\right) :&
a+c 
\end{array}\\
\hline
|e|\leq a
& \begin{array}{cc} \left[-a,-\sqrt{ac}\right) :&
-(\frac{ac}{e}+e) 
\\
\left[-\sqrt{ac},\sqrt{ac}\right] :& \varnothing \\
\left(\sqrt{ac},a\right] :&
-(\frac{ac}{e}+e) 
\end{array}
&  \begin{array}{cc} [-a,a] :& \varnothing \end{array}
  & \begin{array}{cc} [-a,a] :& a+c  \end{array}  \\
\hline
a < e & \begin{array}{cc} (a,\infty) : & -(\frac{ac}{e}+e)  \end{array}
& \begin{array}{cc} (a,\infty)  :& \left\{-a\left(\frac ae+\frac ea\right),2a\right\} \end{array}
& \begin{array}{cc} (a,\sqrt{ac}] :& a+c\\
\left(\sqrt{ac},\infty\right) :& \left\{-\left(\frac{ac}{e}+e\right), a+c \right\} \end{array}\\
\hline
\end{array}
\end{equation*}           }
\label{thm:decentralized}
\end{theorem}

\begin{proof} We need to check in each case of the above table, which case of the
appropriate Theorem in Section \ref{chap:spectrum} applies.
The result of that process is given in the table below.
Each entry is a list of domains for e to the left of the colon together with the
subcases of the relevant Theorem (to the right of the colon).
Once we know which case applies, we list the appropriate special eigenvalues (if any)
specified by those Theorems and Lemma \ref{lem:special-evals} and that gives the table
in the Theorem.
{\small
\begin{equation*}
\begin{array}{|c||c|c|c|}
\hline
---& 0<c<a & c=a & a<c \\
\hline\hline
e<-a; \mathrm{\;Theorem\; \ref{thm:e<-a}} & \begin{array}{cc} (-\infty,-a) :& \text{\ref{thm:e<-a}-3} \end{array}
& \begin{array}{cc} (-\infty,-a) :& \text{\ref{thm:e<-a}-3} \end{array}
& \begin{array}{cc} \left(-\infty, -\sqrt{ac}\right) :&\text{\ref{thm:e<-a}-2c}\\
\left[-\sqrt{ac},-a\right) :& \text{\ref{thm:e<-a}-3} \end{array}\\
\hline
|e|\leq a; \mathrm{\;Theorem\; \ref{thm:-a.e.a}}
& \begin{array}{cc} \left[-a,-\sqrt{ac}\right) :& \text{\ref{thm:-a.e.a}-1} \\
\left[-\sqrt{ac},\sqrt{ac}\right] :&\text{\ref{thm:-a.e.a}-2} \\
(\sqrt{ac},a] :& \text{\ref{thm:-a.e.a}-3}         \end{array}
&  \begin{array}{cc} [-a,a] :& \text{\ref{thm:-a.e.a}-2}  \end{array}
  & \begin{array}{cc} [-a,a] :& \text{\ref{thm:-a.e.a}-1}  \end{array}  \\
\hline
a < e ; \mathrm{\;Theorem\; \ref{thm:e>a}} & \begin{array}{cc} (a,\infty) :& \text{\ref{thm:e>a}-3} \end{array}
& \begin{array}{cc} (a,\infty) :& \text{\ref{thm:e>a}-2}  \end{array}
& \begin{array}{cc} [a,\sqrt{ac}] :& \text{\ref{thm:e>a}-1} \\
\left(\sqrt{ac},\infty\right) :& \text{\ref{thm:e>a}-2}  \end{array}\\
\hline
\end{array}
\end{equation*}         }

The verification of the table in this proof is a tedious process. We will outline how to do
that when $e<-a$. For other values of $e$ the process is very similar.

{
Since $A$ is decentralized we have (Definition \ref{def:decentralized}) $d=c-e$.
Recall that $a$ and $c$ are positive. So:
\begin{equation*}
a>0, \quad c>0, \quad e<-a<0, \quad d=c-e
\end{equation*}

Expanding $0\leq (c+e)^2$ and subtracting $4ce$ gives $-4ce \leq
c^2-2ce+c^2 = (c-e)^2$. As $e$ is negative, taking square roots
implies $2\sqrt{c|e|} \leq c-e = d$. This implies that conditions for
Theorem \ref{thm:e<-a} cases \ref{thm:e<-a}-2c or \ref{thm:e<-a}-3
must hold. Case \ref{thm:e<-a}-3 applies if $(a-e)\sqrt{\frac{c}{a}}
  \leq c-e$, using $c=\sqrt{ac}\sqrt{\frac{c}{a}}$ and
  $a\sqrt{\frac{c}{a}} = \sqrt{ac}$, this condition is equivalent to
\begin{equation}\label{eq:row1}
-\sqrt{ac}\left(\sqrt{\frac{c}{a}} - 1\right) \leq e \left(\sqrt{\frac{c}{a}} -1\right) .
\end{equation}

If $c<a$ (corresponding to the first column of the above tables),
$\sqrt{\frac{c}{a}}-1<0$ so (\ref{eq:row1}) holds if $-\sqrt{ac} \geq
e$, so case \ref{thm:e<-a}-3 holds for all $e\leq
-\sqrt{ac}$. However, the appropriate set for $e$ is actually smaller
as we already have restricted e to $e<-a$, and $-a < -\sqrt{ac}$ (as
$c<a$). So case \ref{thm:e<-a}-3 holds for $e\in(-\infty,-a)$.

If $c=a$ (corresponding to the second column of the above tables),
(\ref{eq:row1}) reduces to $0\leq 0$ which is true, implying that
case \ref{thm:e<-a}-3 holds for all $e\in (-\infty,-a)$.

If $c>a$ (corresponding to the third column of the above
tables),$\sqrt{\frac{c}{a}}-1>0$ so (\ref{eq:row1}) holds if
$-\sqrt{ac}\leq e$. This implies that case \ref{thm:e<-a}-3 holds for
$e\in (-\sqrt{ac},-a)$, (which is not vacuous as $\sqrt{ac}<a$ for
$a<c$), otherwise case \ref{thm:e<-a}-2c holds for
$e\in(-\infty,-\sqrt{ac})$.

When translating from the table of conditions to the results for the
``special eigenvalues'', note that the case \ref{thm:e<-a}-3 produces
the eigenvalue $r_+$, whose expression depends on the sign of
$c+e$. It is straightforward to show that $c+e$ must be negative if
$e<-a$ and $c\leq a$ (accounting for the first two columns of the
above tables), and that $c+e$ is positive if $a<c$ but $e \in
[-\sqrt{ac},-a)$ (accounting for the second case in the third column
of the tables above). Note that the sign of $c+e$ is not determined
for $e\in (-\infty,-\sqrt{ac})$, but this does not affect the
eigenvalues in this case as case \ref{thm:e<-a}-2c produces both $r_+$
and $r_-$ as special eigenvalues.

}
This finishes the classification of the decentralized spectra when $e<-a$. The
strategy when $e\geq -a$ is the same.
\end{proof}

A special case of this result, namely $\rho\in(0,1)$ and $a=1-\rho$, $c=\rho$, $e=\rho$
and $d=0$, a was proved in \cite{FV}.

\section{Applications to Decentralized Systems of Differential Equations}
\label{chap:differential equations}

In this section we will take up the asymptotic stability of the
consensus (first order) and flocking (second order) given in Equations
\ref{eq:first-order} and \ref{eq:second-order}. In particular we prove
that for both of these systems asymptotic stability does not depend on
boundary conditions if these are `reasonable' (in this case $|e|<a$).

The first order system has an eigenvalue 0, and the second order
system has an eigenvalue 0 of multiplicity at least 2. These
eigenvalues are associated with the solutions given in Equation
\ref{eq:coherent solutions}.  The systems are called asymptotically
stable if those eigenvalues have multiplicities exactly 1 and 2,
respectively, and if all other eigenvalues have negative real part.

\begin{corollary} Let $a$, $c$, $d$, $e$ fixed, and so that $A$ is
  decentralized.  Then for $n$ large enough the system in Equation
  \ref{eq:first-order} is asymptotically stable if $a+e>0$, and
  asymptotically unstable if $a+e<0$ and $c+e\neq 0$.
\label{cory:first-order}
\end{corollary}

\begin{proof}
  The $n+1$ eigenvalues of $L$ associated with this systems are
  obtained by subtracting $a+c$ from those given by Theorem
  \ref{thm:decentralized}. One eigenvalue equals 0.  Most other
  eigenvalues are given by $2\sqrt{ac} \cos \psi_\ell-(a+c)<0$. There
  are at most 2 eigenvalues left and they are exponentially close to
  $0$ or to $-\frac{ac}{e}-e-a-c=-\frac{(a+e)(c+e)}{e}$.

  When $a+e<0$ and $c+e\neq 0$ then we must have $e<-a$, and so are
  among the cases in the top row of the table in Theorem
  \ref{thm:decentralized}.  If $c+e<0$, then the (approximate)
  eigenvalue $-\frac{(a+e)(c+e)}{e}$ (which always appears in these
  cases) is greater than 0, implying asymptotic instability. If
  $c+e>0$, then the parameters satisfy $a < | e| < c$, so we are in
  the cases (top right of table in theorem \ref{thm:decentralized})
  where $a+c$ is an approximate eigenvalue of $A$. Then as $a+e<0$,
  Corollary \ref{cory2:App2} implies that the actual eigenvalue of $A$
  that $a+c$ approximates is greater than $a+c$, implying that there
  is a positive eigenvalue for $L$.

  When $a+e>0$ the approximate eigenvalue $-\frac{(a+e)(c+e)}{e}$ (if
  it occurs) is less than zero. This can be seen by noting that
  $-\frac{(a+e)(c+e)}{e} >0$ only if $c+e$ and $e$ have opposite
  signs, i.e. if $e<0$ and $|e|<c$. Looking at the table in Theorem
  \ref{thm:decentralized}, this is only possible for cases in the
  middle row. The only such case where $-\frac{(a+e)(c+e)}{e}$ is implied
  as an eigenvalue of $L$ is when $e\in [-a,-\sqrt{ac}]$ and $c<a$, but
  we cannot have $e\in [-a,-\sqrt{ac}]$ as $|e|<c<\sqrt{ac}$.
  The only other potential cause of instability is the eigenvalue of
  $L$ that is asymptotically equal to 0, however if $a+e>0$ Corollary
  \ref{cory2:App2} implies that the this eigenvalue is actually
  slightly less than 0.
\end{proof}

A special case of the following result was first proved in \cite{TVS}

\begin{corollary} Let $a$, $c$, $d$, $e$ fixed, and so that $A$ is decentralized.
In order for the second order system of Equation \ref{eq:second-order} to be asymptotically
stable, $\alpha$ and $\beta$ must be negative. If  $\alpha$ and $\beta$ are positive
then for $n$ large enough the system is asymptotically stable if $a+e>0$ and asymptotically unstable if $a+e<0$ and $c+e\neq 0$. If $\alpha$ or $\beta$ are positive, the system is unstable.
\label{cory:second-order}
\end{corollary}

\begin{proof} The eigenvalue equation for second order system can be written as follows:
\begin{equation*}
\left(\begin{array}{c}
\dot z \\
\ddot z \\
\end{array}\right)=
\left(\begin{array}{cc}
0 & I \\
-\alpha L & -\beta L \\
\end{array}\right)
\left(\begin{array}{c}
z \\
\dot z \\
\end{array}\right)=
\left(\begin{array}{c}
\nu z \\
\nu \dot z \\
\end{array}\right)
\end{equation*}
The second equality yields 2 equations. The first of these is that $\dot z=\nu z$.
Suppose $\lambda$ is an eigenvalue of $-L$. Each eigenvalue $\lambda$ gives rise
to two eigenvalues $\nu_\pm$, because if we substitute it and $\dot z=\nu z$ into the second
of the above equations, we see that
\begin{equation*}
(\nu^2-\beta \lambda \nu - \alpha)z=0\quad \Rightarrow \quad
\nu_\pm =\frac 12\left( \beta \lambda \pm \sqrt{\beta^2\lambda^2+4\alpha\lambda}\right)
\end{equation*}
Corollary \ref{cory:first-order} says that in all cases there are negative $\lambda$.
Thus both $\alpha$ and $\beta$ must be positive if the system is stable. Assuming
that $\alpha$ and $\beta$ are positive, we have stability precisely in those cases
where Corollary \ref{cory:first-order} insures stability for the first order system.
\end{proof}

A few observations are in order here. Asymptotic stability is not the whole story.
In fact as $n$ becomes large, even for asymptotically stable systems the transients
in Equations \ref{eq:first-order}  and \ref{eq:second-order} may grow exponentially in $n$.
This is due to the fact the eigenvectors are not normal and a dramatic example
of this was given in \cite{positionpaper}. Here we can see it expressed
in the form of the eigenvectors given by Equation \ref{eq:eigenpair} of Lemma \ref{lem:polynomial}:
if $\tau\neq 1$ the eigenvectors have an exponential behavior. When $\tau$ is small
a long time will pass before a change in the velocity of the leader is felt
at the back of the flock, and so coherence will be lost. On the other hand when
$\tau$ is large a change will immediately amplify exponentially towards the back of the flock.
These observations have been proved for $e=1/2$ and $d=0$ by \cite{TVS}.
But to address that problem in more generality, different concepts are needed.
In \cite{Cantos2}, assuming some conjectures, we show that this phenomenon
indeed appears to be independent of the boundary conditions ($e$ and $d$). In that
paper we also show that the behavior of the system \ref{eq:second-order} can be substantially
improved if the position Laplacian and the velocity Laplacian are allowed to be different.

\section{Appendix 1}
\label{chap:App1}

We note that for all positive integers $n$:
\begin{equation*}
n\left| \mathrm{Im}(e^{i\phi}) \right|\geq \left| \mathrm{Im}(e^{in\phi}) \right|
\end{equation*}
and equality holds if and only if $\sin(\phi)=0$. This immediately implies that
\begin{eqnarray*}
\forall\, \phi\in(0,\pi) & n\sin \phi > \sin n \phi & \\
\forall\, \phi\in(-\pi, 0) & n\sin \phi < \sin n \phi &
\end{eqnarray*}
The following stronger result is also a consequence of this:

\begin{corollary}
For all $B\in(-\infty,1]$ and all positive integers $n$ we have that
\begin{eqnarray*}
\forall\, \phi\in(0,\pi) &
\dfrac{\partial}{\partial \phi} \left(\cot n\phi\,\sin \phi- B \cos \phi\right)>0& \\
\forall\, \phi\in(-\pi,0) &
\dfrac{\partial}{\partial \phi} \left(\cot n\phi\,\sin \phi- B \cos \phi\right)<0&
\end{eqnarray*}
\label{cor:App1}
\end{corollary}

\begin{proof} It is sufficient to prove the first case. So let $\phi \in (0,\pi)$.
\begin{eqnarray*}
\dfrac{\partial}{\partial \phi} \left(\cot n\phi\,\sin \phi- B \cos \phi\right)&=&
\dfrac{\cos \phi \,\cos n\phi \,\sin n\phi-n\sin \phi+B\sin \phi \, \sin n \phi \,\sin n\phi}{(\sin n \phi)^2}\\
 & \leq & \frac{(\cos \phi\, \cos n\phi+\sin \phi \,\sin n \phi) \sin n \phi - n \sin \phi}{(\sin n \phi)^2}\\
 & = & \dfrac{\cos(n-1)\phi\, \sin\phi}{(\sin n \phi)^2}\\
 & \leq & \dfrac{\sin n \phi- n \sin \phi}{(\sin n \phi)^2}<0
\end{eqnarray*}
The inequalities hold wherever $\sin n \phi \neq 0$.
\end{proof}

It follows that the function $\cot n \phi \sin \phi$ is decreasing on $(0,\pi)$ (wherever it is defined)
and each branch intersects the function $A+B\cos \phi$ at most once as long as $B\in (-\infty,1]$.
In fact Proposition \ref{prop:e<-a} shows that the same is true for \emph{every} given $B$
\emph{as long as $n$ is large enough}. But a direct proof seems hard.

\section{Appendix 2}
\label{chap:App2}

The results in this paper are based on the roots of the
polynomial $(ay^2-d\tau y - e) y^{2n} + (ey^2+d\tau y - a)$, from
equation (\ref{eq:polynomial}). The following results show how at
least one root is asymptotically close to a root of $ay^2-d\tau y - e
= 0$.

Below, let $F_t(z)=p(z)z^{m}+tq(z)$ where $p$, and $q$ are polynomials
of degree 2 such that $p$ has a simple root at $z_0$ with $|z_0|=r>1$
and $q(z_0)\neq 0$.

\begin{lemma}
For $m$ large enough there is a root $z_1$ of $F_1$ given above and a positive constant $K$ such that
$|z_1-z_0|\leq Kr^{-m}$.
\label{lem:App2}
\end{lemma}

\begin{proof} For $K>0$ to be determined later, let
\begin{equation*}
B_m\equiv \{z\,|\,|z-z_0|\leq Kr^{-m}\}
\end{equation*}
There are positive constants $A$, $B$, $C$, $D$, and $\epsilon$ such that for all $z\in B_m$ we have that
for $m$ large enough
\begin{equation*}
|q'(z)|\leq A \;\mathrm{and}\; |p'(z)| \geq B+\epsilon
\;\mathrm{and}\; |p(z)|\leq Cr^{-m} \; \mathrm{and}\; |q(z)|\leq D .
\end{equation*}
Note that $z_0$ is a root of $F_0$.
As $\frac{\partial F_t(z)}{\partial z} = p'(z) z^m + p(z)mz^{m-1} + t
q'(z)$, we have
\begin{align*}
\left|\frac{\partial F_t(z)}{\partial z} \right| & \geq
|z|^m(B+\epsilon) - |z|^{m-1}mCr^{-m} - A
\end{align*}
Using $r-Kr^{-m} \leq |z| \leq Kr^{-m} + r$, which holds for $z\in
B_m$, shows
\begin{equation*}
\left|\frac{\partial F_t(z)}{\partial z} \right| \geq (r-Kr^{-m})^m (B+\epsilon) - \frac{1}{r}(1+Kr^{-m-1})^{m-1} m C - A
\end{equation*}
As $r>1$, this shows that $\left|\frac{\partial F_t(z)}{\partial
    z} \right| > (r-Kr^{-m})^m B$ for sufficiently large $m$, provided $z\in B_m$.

In particular, $\left(\frac{\partial F_t(z)}{\partial z}\right)^{-1}$ exists,
so the implicit function theorem implies that there is a root
$z_t$ of $F_t$ with $z_{t=0}=z_0$ and that satisfies:
\begin{equation*}
\dfrac{\partial z_t}{\partial t} =-\dfrac{\frac{\partial F_t(z)}{\partial t}}{\frac{\partial F_t(z)}{\partial z}} =
\dfrac{-q}{\frac{\partial F_t(z)}{\partial z}}
\end{equation*}
As long as $z_t$ in $B_m$ we have, using the bounds on $q$ and $\frac{\partial F_t(z)}{\partial z}$,
\begin{equation*}
\left|\dfrac{\partial z_t}{\partial t}\right|\leq \dfrac DB \,\left(r-Kr^{-m}\right)^{-m}\leq
\dfrac{2D}{B}\,r^{-m} ,
\end{equation*}
where the second inequality will hold for sufficiently large $m$.
Thus $|z_t-z_0|$ is bounded by $\left|\int_0^t\,\frac{\partial z_t}{\partial t}\,dt\right|\leq
\int_0^t\,\left|\frac{\partial z_t}{\partial t}\right|\,dt \leq t \dfrac{2D}{B}\,r^{-m} $. If $K$
is taken equal to $\frac{2D}{B}$, this ensures that $z_t \in B_m$,
letting $t$ equal 1 then gives the desired result.
\end{proof}

\begin{corollary}
  In the above Lemma, let $p(z)=az^2-d\tau z-e$ and suppose $p$ has
  two real roots $\mu$ and $\nu$ such that $\nu=z_0$ satisfies
  $|\nu|>1$, and $\nu>\mu$. Suppose furthermore that
  $q(z)=-z^2p(z^{-1})$.  For $m$ large enough there is a root $z_1$ of
  $F_1$ given above and a positive constant $K$ such that
  $|z_1-z_0|\leq Kr^{-m}$. Moreover provided $a+e\neq 0$ we have
\begin{equation*}
\mathrm{sgn}(z_1-z_0)=-\mathrm{sgn}(a+e)
\end{equation*}
\label{cory:App2}
\end{corollary}

\begin{proof} Note that $a+e\neq 0$ implies that $p(z_0^{-1})\neq 0$.
The first part of the corollary now of course follows from Lemma \ref{lem:App2}.
So it suffices to determine
\begin{equation*}
\mathrm{sgn}(z_1-z_0)=\mathrm{sgn}\left(\dfrac{\partial z_t}{\partial t}\right)
\end{equation*}
where $z_0=\nu$. According to the proof of the last lemma, for $m$ is big enough,
we may take into account only the leading term, and so since $\nu>1$ we get:
\begin{equation*}
\mathrm{sgn}(z_1-z_0)=\mathrm{sgn}\left(\dfrac{-q}{(zp'+mp)z^{m-1}+tq'}\right)=
\mathrm{sgn}\left(\dfrac{\nu^2 p(\nu^{-1})}{p'(\nu)\nu^{m}}\right)=
\mathrm{sgn}\left(\dfrac{ p(\nu^{-1})}{p'(\nu)}\right)
\end{equation*}
Finally
\begin{equation*}
\mathrm{sgn}\left(\dfrac{ p(\nu^{-1})}{p'(\nu)}\right)=
\mathrm{sgn}\left(\dfrac{(\nu^{-1}-\nu)(1-\nu\mu)}{\nu-\mu}\right)=
\mathrm{sgn}\left(\dfrac{(1-\nu^2)(1+\frac ea)}{\nu-\mu}\right) =
-\mathrm{sgn}(a+e)
\end{equation*}
which proves the result.
\end{proof}

\begin{corollary}
  Under the conditions of  Corollary \ref{cory:App2}, the eigenvalue expressions
  $r_0=\sqrt{ac}(z_0+z_0^{-1})$ and $r_1=\sqrt{ac}(z_1+z_1^{-1})$
  satisfy
\begin{equation*}
\mathrm{sgn}(r_1-r_0)=-\mathrm{sgn}(a+e)
\end{equation*}
\label{cory2:App2}
\end{corollary}
\begin{proof}
This follows immediately as the function $f(x)=\sqrt{ac}(x+x^{-1})$ is increasing
for $|x| > 1$, and because $|z_0| > 1$.
\end{proof}

\bibliographystyle{siam}
\bibliography{tridiagonal2}












\end{document}